\ifdefined\pdfminorversion
  \pdfoutput=1
  \input glyphtounicode
  \ifdefined\pdfinterwordspaceon
    \pdfinterwordspaceon
  \fi
\fi

\documentclass[justified,notoc,twoside,nobib]{tufte-handout}
\usepackage[numbers,square]{natbib}
\setcitestyle{numbers,square}

\usepackage[T1]{fontenc}
\usepackage[utf8]{inputenc}
\usepackage[english]{babel}
\ifdefined\pdfminorversion
  \usepackage[tracking]{microtype}
\else
  \usepackage{microtype}
\fi
\usepackage[strict]{changepage}
\usepackage{amsmath}
\usepackage{amsthm}
\usepackage{thmtools}
\usepackage{hyperref}
\usepackage{xurl}
\usepackage{beton}
\usepackage{sectsty}
\allsectionsfont{\fontfamily{LinuxBiolinumT-OsF}\selectfont}
\usepackage{amssymb}
\usepackage{mathtools}
\usepackage{cleveref}
\usepackage{needspace}
\usepackage[osf,sc]{mathpazo}

\declaretheorem[name=Theorem,numberwithin=section]{thm}
\declaretheorem[name=Lemma,sibling=thm]{lem}
\declaretheorem[name=Corollary,sibling=thm]{cor}
\crefname{thm}{theorem}{theorems}
\Crefname{thm}{Theorem}{Theorems}
\crefname{lem}{lemma}{lemmas}
\Crefname{lem}{Lemma}{Lemmas}
\crefname{cor}{corollary}{corollaries}
\Crefname{cor}{Corollary}{Corollaries}

\AtBeginDocument{\fontsize{12}{15}\selectfont\fancyhfoffset{0pt}}
\hypersetup{
  colorlinks=true,
  linkcolor=black,
  citecolor=black,
  urlcolor=black
}

\newcommand{\Ftwo}{\mathbb F_2}
\newcommand{\Om}{\Omega}
\newcommand{\card}[1]{\lvert #1\rvert}
\newcommand{\set}[1]{\left\{#1\right\}}
\newcommand{\pair}[2]{\left\langle #1,#2\right\rangle}

\title[GP III.1]
{Graph Puzzles III.1:\\
A proof of Sabidussi's compatibility conjecture with 4 colours}
\date{}
\renewcommand{\plainauthor}{Nikolay Ulyanov}
\ifdefined\XeTeXversion
\fi
\hypersetup{
  pdftitle={Graph Puzzles III.1: A proof of Sabidussi's compatibility conjecture with 4 colours},
  pdfauthor={Nikolay Ulyanov},
  pdfsubject={A motivated construction of compatible circuit decompositions},
  pdfkeywords={Eulerian multigraph, circuit decomposition, transition system, parity, edge-colouring}
}

\begin{document}
\maketitle

\noindent
Nikolay Ulyanov\textsuperscript{*}\hfill
\href{mailto:ulyanick@gmail.com}{\textsuperscript{*}\nolinkurl{ulyanick@gmail.com}}

\begin{abstract}
\noindent
\newthought{Abstract: }We prove Sabidussi's compatibility conjecture. Let $G$ be a finite connected multigraph in which every vertex has even degree and the minimum degree is at least four, and let $T$ be an Euler tour of $G$. The edges of $G$ can be partitioned into circuits (connected $2$-regular subgraphs) so that no circuit contains two edges used consecutively anywhere in $T$. In fact, the edges can be four-coloured so that every such pair receives different colours and every colour class has even degree at every vertex. We use a counting argument based on the Chevalley--Warning theorem to show that a four-colouring with the required properties exists. Splitting each colour class into circuits then gives the desired compatible decomposition.
\end{abstract}

\noindent\rule{\linewidth}{0.4pt}

\section{Introduction}\label{sec:introduction}

All graphs in this paper are finite multigraphs; parallel edges
are allowed. An edge has two distinct \emph{half-edges}, each incident
with a vertex. A graph is \emph{even} if every vertex has even degree, and
\emph{Eulerian} if it is connected and even. A \emph{circuit} is a
nonempty connected \(2\)-regular subgraph. A \emph{circuit decomposition}
is a partition of the edge set into circuits.

A transition at a vertex is an unordered pair of distinct half-edges
incident with that vertex. A transition system pairs the half-edges at
each vertex. A circuit pairs its two half-edges at each vertex that it
meets; it is \emph{compatible} with a transition system if none of these
pairs is a prescribed transition. An Euler tour induces a transition
system by pairing the arriving and departing half-edges at each passage
through a vertex.

We prove the following strengthening of Sabidussi's conjecture.

\begin{thm}\label{thm:sabidussi}
Let \(G\) be a finite Eulerian multigraph with minimum degree
\(\delta(G)\ge4\), and let
\[
  T=(e_0,e_1,\ldots,e_{m-1}),
  \qquad m=\card{E(G)},
\]
be an Euler tour, with indices read modulo \(m\). There is a colouring
\(\chi:E(G)\to\Ftwo^2\) such that
\begin{enumerate}
\item \(\chi(e_{i-1})\ne\chi(e_i)\) for every \(i\);
\item every vertex has even degree in each colour class.
\end{enumerate}
Consequently, \(G\) has a circuit decomposition compatible with the
transition system induced by \(T\).
\end{thm}

The conjecture, due to Gert Sabidussi, was recorded by Fleischner in
\citep{Fleischner1980}, where he also proved it for planar graphs.
Fan and Zhang proved the conjecture for \(K_5\)-minor-free graphs
\citep{FanZhang2000}. Fleischner also showed that it suffices to
consider graphs in which every vertex has degree four or six
\citep[Lemma~1, p.~236]{Fleischner1984}.
The construction below treats every even degree directly,
however it was inspired by the aforementioned reduction.

The proof constructs the edge colouring asserted in \cref{thm:sabidussi}.
We begin with its two requirements: consecutive tour edges must have
different colours, and each colour class must have even degree at every
vertex. We express the second requirement as equations over \(\Ftwo^2\),
then use these equations together with the first requirement to develop
the construction. A count modulo two using the Chevalley--Warning theorem
will show that the required choices can always be made.

The theorem has also been formalized in Lean. The formalization is
available in the author's GitHub repository:
\href{https://github.com/gexahedron/sabidussi-lean}
{\nolinkurl{https://github.com/gexahedron/sabidussi-lean}}.

\subsection{Statement of AI use}
The proof in this paper is essentially due to gpt 5.6 pro, produced by 2 iterations of a prompt similar to \citep{OpenAI2026CDCPrompt}, with a help of a couple of additional hints.

\subsection{Acknowledgements}
The author thanks Tom de Groot for valuable suggestions on the exposition
of the proof.

\section{From compatible circuits to colour parity}\label{sec:colour-parity}

We fix the finite Eulerian multigraph \(G\), with $m=\card{E(G)}$ edges, and with an Euler tour $T=(e_0,e_1,\ldots,e_{m-1})$. We begin by explaining the two requirements on the colouring.

\begin{lem}\label{lem:colouring-criterion}
Suppose that an edge colouring has even degree in each colour class at
every vertex, and assigns different colours to consecutive edges of
\(T\). Then the colour classes can be decomposed into circuits that are
compatible with the transition system induced by \(T\).
\end{lem}

\begin{proof}
Every finite even multigraph decomposes into circuits: take an Euler tour
in each component containing an edge and split it at repeated vertices
until the resulting closed trails are circuits. Apply this separately to
each colour class. The resulting circuits partition \(E(G)\).
Each is monochromatic, whereas the two half-edges of every prescribed
transition belong to differently coloured edges. Thus none of its
transitions is prescribed.
\end{proof}

For a graph in which every vertex has degree four, alternating two
colours along \(T\) already works. The tour has even length, and each of
the two passages at a vertex contributes one half-edge of each colour.
At a degree-six vertex, the same alternation would contribute three
half-edges of each colour, so it would fail the even-degree requirement.
We use four colours, represented by the vectors in \(\Ftwo^2\), to
obtain more local choices.

All vector calculations from now on take place over \(\Ftwo\).
If every colour occurs evenly among the half-edges at a vertex, their
vector sum is zero. The converse fails: the four distinct vectors
$
  (0,0), (1,0), (0,1), (1,1)
$
sum to zero, although each occurs once. We need one further bit to
distinguish this possibility from even multiplicities. The function
\[
  q:\Ftwo^2\longrightarrow\Ftwo,
  \qquad q(a,b)=ab,
\]
provides it: \(q\) is one exactly at the colour \((1,1)\).

\begin{lem}[Four-colour parity]\label{lem:four-colour-parity}
Let \((z_h)_{h\in H}\) be a family of vectors in \(\Ftwo^2\), where
\(\card H\) is even. Every colour occurs an even number of times if
and only if
\begin{equation}\label{eq:four-colour-conditions}
  \sum_{h\in H}z_h=0,
  \qquad
  \sum_{h\in H}q(z_h)=0.
\end{equation}
\end{lem}

\begin{proof}
Even multiplicities imply both equations. Conversely, let
\(n_{ab}\in\Ftwo\) be the parity of the multiplicity of \((a,b)\).
The even cardinality and the two coordinates of the vector sum give
\[
  n_{00}+n_{10}+n_{01}+n_{11}=0,\qquad
  n_{10}+n_{11}=0,\qquad n_{01}+n_{11}=0.
\]
The quadratic sum is \(n_{11}\). If it is zero, these equations
successively give \(n_{10}=n_{01}=n_{00}=0\).
\end{proof}

Thus our colouring problem has three precise requirements: consecutive
colours must differ, the vector sum at each vertex must vanish, and the
quadratic sum at each vertex must vanish.

\section{Difference labels and the remaining obstruction}\label{sec:obstruction}

\subsection{Local conditions and provisional integration}

Let \(v_i\) be the vertex at which \(T\) passes from \(e_{i-1}\) to
\(e_i\), and put
\[
  I_v=\set{i\in\{0,\ldots,m-1\}:v_i=v}.
\]
The passages indexed by \(I_v\) pair all half-edges at \(v\), hence
$
  \card{I_v}=\dfrac{\deg_G(v)}2\ge2.
$
We choose a starting point on \(T\), so that comparisons \(j<i\)
use the ordinary linear order \(0<1<\cdots<m-1\).

Write \(x_i\in\Ftwo^2\) for the intended colour of \(e_i\).
At passage \(i\), we require \(x_{i-1}\ne x_i\). Over \(\Ftwo^2\),
this is equivalent to \(x_{i-1}+x_i\ne0\), so we introduce the
\emph{difference label} (we use "$+$" sign instead of "$-$" because over \(\Ftwo^2\) they are equivalent):
\[
  y_i=x_{i-1}+x_i.
\]
Each \(y_i\) records the change in colour between the two edges used at
passage \(i\). The distinctness requirement is now simply \(y_i\ne0\).
The same variables also express the vector-sum requirement: the sum of
the incident half-edge colours at \(v\) is
\[
  \sum_{i\in I_v}(x_{i-1}+x_i)=\sum_{i\in I_v}y_i.
\]
The first two colouring requirements therefore suggest choosing labels
with
\begin{equation}\label{eq:local-difference-conditions}
  y_i\ne0\quad(0\le i<m),
  \qquad
  \sum_{i\in I_v}y_i=0\quad(v\in V(G)).
\end{equation}
This choice is local: each vertex assigns one label to each of its
passages, subject to a zero-sum condition.
We will choose these differences first and recover the edge colours
by following the tour. For the recovery to be consistent, the colour
must return to its starting value after one full traversal. The local
zero sums guarantee precisely this.

\begin{lem}[Integrating the differences]\label{lem:integration}
Suppose the labels satisfy \cref{eq:local-difference-conditions}.
Set \(x_{-1}=0\) and
\begin{equation}\label{eq:prefix-colours}
  x_i=\sum_{j=0}^{i}y_j\qquad(0\le i<m).
\end{equation}
Then \(x_{m-1}=x_{-1}\), consecutive tour edges have different
colours, and the sum of the incident half-edge colours is zero at
every vertex.
\end{lem}

\begin{proof}
The sets \(I_v\) partition the passage indices, so the local sums give
\(\sum_{i=0}^{m-1}y_i=0\). Thus \(x_{m-1}=0=x_{-1}\), and the
colours close around the tour. The prefix formula gives
\(x_{i-1}+x_i=y_i\ne0\), including at the passage with index zero.
Finally, the incident half-edge colours at \(v\) are
\((x_{i-1},x_i)_{i\in I_v}\); their sum is the local sum of the
\(y_i\)'s, which vanishes.
\end{proof}

The chosen differences therefore produce an edge colouring with
distinct consecutive colours and zero incident vector sums. It remains
to choose them so that each colour occurs evenly at every vertex.

\subsection{The quadratic obstruction}

The remaining condition is that the number of incident half-edges
coloured \((1,1)\) should be even at every vertex. Since \(q\) is one exactly
at this colour, its parity at \(v\) is
$
  \sum_{i\in I_v}\bigl(q(x_{i-1})+q(x_i)\bigr).
$
The \cref{lem:four-colour-parity} says that every colour class has even
degree at \(v\) if and only if this sum is zero.
Our goal is to rewrite it in terms of the difference labels \(y_i\) instead of the colours \(x_i\).

For \(a=(a_1,a_2)\) and \(b=(b_1,b_2)\) in \(\Ftwo^2\), introduce the \emph{polar form}
\[
  \pair{a}{b}=a_1b_2+a_2b_1.
\]
This form is bilinear, symmetric, and alternating. Expanding
\(q(a+b)\) gives the polarization identity
\begin{equation}\label{eq:polarization}
  q(a+b)=q(a)+q(b)+\pair{a}{b}.
\end{equation}
Using \(x_i=x_{i-1}+y_i\) and the prefix formula
\cref{eq:prefix-colours}, we obtain
\[
  q(x_{i-1})+q(x_i)
  =q(y_i)+\pair{x_{i-1}}{y_i}
  =q(y_i)+\sum_{0\le j<i}\pair{y_i}{y_j}.
\]
Summing over \(i\in I_v\), we separate the earlier passages at \(v\)
from those at other vertices:
\begin{align*}
  \sum_{i\in I_v}\bigl(q(x_{i-1})+q(x_i)\bigr)
  &=\sum_{i\in I_v}q(y_i)
    +\sum_{\substack{i,j\in I_v\\j<i}}\pair{y_i}{y_j}\\
  &\quad+\sum_{u\ne v}
    \sum_{\substack{i\in I_v,\ j\in I_u\\j<i}}\pair{y_i}{y_j}.
\end{align*}
The first two sums involve only labels at \(v\). Repeated use of
\cref{eq:polarization} combines them into the quadratic value of their
sum, which vanishes by the local zero-sum condition:
\begin{equation}\label{eq:finite-polarization}
  \sum_{i\in I_v}q(y_i)
  +\sum_{\substack{i,j\in I_v\\j<i}}\pair{y_i}{y_j}
  =q\left(\sum_{i\in I_v}y_i\right)=q(0)=0.
\end{equation}

Only terms pairing a passage at \(v\) with a passage at another vertex
remain. For distinct vertices \(v,u\), call their contribution the
\emph{interaction}
\begin{equation}\label{eq:label-interaction}
  \beta_{vu}(y)=
  \sum_{\substack{i\in I_v,\ j\in I_u\\j<i}}\pair{y_i}{y_j}.
\end{equation}
For each vertex \(v\), define its total interaction with the other
vertices by
\begin{equation}\label{eq:label-obstruction}
  Q_v(y)=\sum_{u\ne v}\beta_{vu}(y).
\end{equation}
With this notation, the regrouping above gives
\begin{equation}\label{eq:parity-detector}
  \sum_{i\in I_v}\bigl(q(x_{i-1})+q(x_i)\bigr)=Q_v(y).
\end{equation}
We have therefore expressed the parity of the incident \((1,1)\)-coloured
half-edges entirely in terms of the difference labels. The remaining
requirements are the \emph{balancing equations}
\begin{equation}\label{eq:remaining-goal}
  Q_v(y)=0\qquad(v\in V(G)).
\end{equation}

%If we join \(u\) and \(v\) when \(\beta_{vu}(y)=1\), this condition says
%that the resulting auxiliary undirected graph has even degree at every
%vertex. This is the sense in which the local choices must be
%\emph{balanced}.

\section{Local assignments at each vertex}\label{sec:patterns}

The preceding section reduced the colouring problem to choosing
difference labels \(y_i\) satisfying three requirements:
\begin{align*}
  y_i&\ne0 &&(0\le i<m),\\
  \sum_{i\in I_v}y_i&=0 &&(v\in V(G)),\\
  Q_v(y)&=0 &&(v\in V(G)).
\end{align*}
Nonvanishing makes consecutive tour edges have different colours.
The local zero sums ensure that integration closes around the tour
and that the incident vector sums vanish. The conditions \(Q_v(y)=0\)
then give even multiplicities of all four colours.

We will give each vertex three assignments of its passage labels,
each satisfying the first two requirements. The remaining task is
to choose one assignment at each vertex so that the third requirement
holds everywhere.

The labels in each assignment will be linear functions of a state
\(t_v\in\Ftwo^2\), with its three nonzero values giving the admissible
choices. This makes the interactions bilinear and each \(Q_v\) a
quadratic polynomial in the state coordinates. The count in
\cref{sec:balancing} will use Chevalley--Warning to show that the states
can be chosen to meet the third requirement.

\subsection{The local patterns}

For the local construction, put
\[
  \Om=\Ftwo^2\setminus\{0\},
  \qquad
  \rho(r,s)=(s,r+s).
\]
The map \(\rho\) cycles the three elements of \(\Om\), and
\begin{equation}\label{eq:rho-sum}
  t+\rho t+\rho^2t=0\qquad(t\in\Ftwo^2).
\end{equation}
For two passages, the labels must be nonzero and sum to zero,
so they have the form \((t,t)\), with \(t\in\Om\). For three
passages, they must be the three distinct
nonzero vectors. We use their three cyclic arrangements
\[
  (t,\rho t,\rho^2t)\qquad(t\in\Om).
\]
Thus both degrees four and six admit a chosen family of three patterns.
At higher even degrees, we add pairs of equal labels to these patterns.
This implements the same construction directly at each vertex.

Write \(I_v=\{i_1<\cdots<i_k\}\), where \(k\ge2\). For every
\(t\in\Ftwo^2\), define
\(\Delta_{v,t}:I_v\to\Ftwo^2\) as follows. If \(k\) is even, put
\(\Delta_{v,t}(i_j)=t\) for all \(j\). If \(k\) is odd, put
\[
  \Delta_{v,t}(i_j)=
  \begin{cases}
    t,      &1\le j\le k-2,\\
    \rho t, &j=k-1,\\
    \rho^2t,&j=k.
  \end{cases}
\]
The admissible states are \(t\in\Om\). We extend the definition to
\(t=0\) so that the patterns are linear maps defined on the whole
vector space \(\Ftwo^2\). The polynomial count in \cref{sec:balancing}
will first allow zero states and then exclude them.

A \emph{pattern} is an assignment to the individual passages at a vertex.
The state
\(t_v\) selects the entire assignment
\(\Delta_{v,t_v}\), and the resulting labels will be
\(y_i=\Delta_{v,t_v}(i)\) for \(i\in I_v\).

Quite straightforwardly, we have the following lemma:

\begin{lem}\label{lem:patterns}
The patterns satisfy
\begin{align}
  \Delta_{v,t}(i)&\ne0
    &&(i\in I_v,\ t\in\Om),\label{eq:pattern-nonzero}\\
  \sum_{i\in I_v}\Delta_{v,t}(i)&=0
    &&(t\in\Ftwo^2).\label{eq:pattern-vertex-sum}
\end{align}
For each fixed passage \(i\), the label \(\Delta_{v,t}(i)\) is
linear in the state \(t\).
\end{lem}

\subsection{Interactions between the states}

For distinct \(v,u\), substitute the patterns into the interaction
\(\beta_{vu}(y)\) from \cref{eq:label-interaction}. We write it in
terms of the two states as
\begin{equation}\label{eq:state-interaction}
  \beta_{vu}(t,s)=
  \sum_{\substack{i\in I_v,\ j\in I_u\\j<i}}
    \pair{\Delta_{v,t}(i)}{\Delta_{u,s}(j)}
  \qquad(t,s\in\Ftwo^2).
\end{equation}

Apply same substitution \(y_i=\Delta_{v,t_v}(i)\), \(i\in I_v\), in
\(Q_v(y)\) from \cref{eq:label-obstruction}. We denote the resulting
expression in the state variables by \(Q_v(t)\):
\begin{equation}\label{eq:state-obstruction}
  Q_v(t)=\sum_{u\ne v}\beta_{vu}(t_v,t_u).
\end{equation}
The requirements \(Q_v(y)=0\) thus become \(Q_v(t)=0\): the
remaining colouring obstruction is now a condition on the state
choices.

\begin{lem}\label{lem:interactions}
The maps \(\beta_{vu}\) are bilinear and satisfy
\begin{equation}\label{eq:interaction-symmetry}
  \beta_{vu}(t,s)=\beta_{uv}(s,t)
  \qquad(t,s\in\Ftwo^2).
\end{equation}
\end{lem}

\begin{proof}
Linearity of the patterns in their states and bilinearity of the pairing
give bilinearity.

For transpose symmetry, each passage belongs to exactly one vertex,
so \(I_v\) and \(I_u\) are disjoint when \(v\ne u\). Every pair of
passages \(i\in I_v\), \(j\in I_u\) therefore occurs in exactly one
of the orders \(j<i\) or \(i<j\). Using the symmetry of the pairing,
we obtain
\begin{align*}
  \beta_{vu}(t,s)+\beta_{uv}(s,t)
  &=\sum_{i\in I_v}\sum_{j\in I_u}
    \pair{\Delta_{v,t}(i)}{\Delta_{u,s}(j)}\\
  &=\pair{\sum_{i\in I_v}\Delta_{v,t}(i)}
          {\sum_{j\in I_u}\Delta_{u,s}(j)}=0.
\end{align*}
The last equality uses \cref{eq:pattern-vertex-sum} at both vertices.
\end{proof}

Each vertex now has three admissible nonzero states, and the interactions
extend to bilinear maps with transpose symmetry on \(\Ftwo^2\). We use
this extension to count choices satisfying all the balancing equations
simultaneously.

\section{Applying Chevalley--Warning}\label{sec:balancing}

The Chevalley--Warning theorem \citep{Chevalley1935,Warning1935}
says that polynomials over \(\Ftwo\) in \(n\) scalar variables
have an even number of common zeros when the sum of their
degrees is less than \(n\).

To use this theorem, we must allow each state to range over all of
\(\Ftwo^2\). The added zero states give invalid local patterns, so a
count on the whole vector space alone does not settle our problem.
We will count balanced families with specified states set to zero,
then use inclusion--exclusion to retain only families in which every
state is nonzero.

\Needspace{14\baselineskip}
\begin{lem}\label{lem:bilinear-balancing}
Let \(U\) be a finite set. For each pair of distinct elements
\(u,v\), let
\(\beta_{uv}:\Ftwo^2\times\Ftwo^2\to\Ftwo\) be bilinear and satisfy
\[
  \beta_{uv}(r,s)=\beta_{vu}(s,r).
\]
For a family \(t=(t_u)_{u\in U}\), put
\[
  Q_u(t)=\sum_{v\ne u}\beta_{uv}(t_u,t_v).
\]
The number of families with \(t_u\in\Om=\Ftwo^2\setminus\{0\}\)
and \(Q_u(t)=0\) for every \(u\in U\) is odd. In particular,
such a family exists.
\end{lem}

\begin{proof}
Allow the states to range over \(\Ftwo^2\). By bilinearity, each
\(Q_u\) is a polynomial of degree at most two in their scalar
coordinates. Transpose symmetry gives the polynomial identity
\begin{equation}\label{eq:dependent-balancing-equation}
  \sum_{u\in U}Q_u(t)=0,
\end{equation}
because every unordered pair contributes twice. This makes one
balancing equation redundant.

For \(S\subseteq U\), let \(F(S)\) be the integer number of
balanced families with \(t_u=0\) for \(u\notin S\); states in
\(S\) may be any of the four vectors. The equations outside
\(S\) vanish identically by bilinearity. If \(S\ne\varnothing\),
the remaining equations still sum to zero by
\cref{eq:dependent-balancing-equation}, so one can be omitted.
There are at most \(\card S-1\) quadratic equations in
\(2\card S\) scalar variables, with degree bound
\[
  2(\card S-1)<2\card S.
\]
The inequality verifies the degree condition in Chevalley--Warning,
which implies that \(F(S)\) is even for every
nonempty \(S\). When \(\card S=1\), there are no equations left
and all four states are solutions. For \(S=\varnothing\), only
the all-zero family is possible, and it is balanced, so
\(F(\varnothing)=1\).

Let \(N\) be the number of balanced families with every state
nonzero. Inclusion--exclusion, applied to the conditions \(t_u=0\),
gives
\begin{align*}
  N
  &=\sum_{S\subseteq U}(-1)^{\card U-\card S}F(S)\\
  &\equiv\sum_{S\subseteq U}F(S)
   \equiv F(\varnothing)\equiv1\pmod2.\qedhere
\end{align*}
\end{proof}

\section{Completing the colouring}\label{sec:completion}

All the ingredients now have the roles for which they were constructed.
The admissible assignments give nonzero labels that sum to zero at
each vertex; balancing removes the quadratic obstruction; integration
turns the resulting labels into edge colours.

\begin{proof}[Proof of \cref{thm:sabidussi}]
The interactions in \cref{eq:state-interaction} are bilinear and
symmetric by \cref{lem:interactions}. Applying
\cref{lem:bilinear-balancing} with \(U=V(G)\) gives a family
\(t=(t_v)_{v\in V(G)}\), with \(t_v\in\Om\), such that
\(Q_v(t)=0\) for every \(v\). Put
\[
  y_i=\Delta_{v,t_v}(i)\qquad(i\in I_v).
\]
By \cref{lem:patterns}, these labels satisfy
\cref{eq:local-difference-conditions}. Define \(x_i\) by
\cref{eq:prefix-colours} and set \(\chi(e_i)=x_i\).
\Cref{lem:integration} shows that the colours close around \(T\),
consecutive tour edges have different colours, and the incident
vector sum vanishes at every vertex.

By \cref{eq:state-obstruction}, \(Q_v(y)=Q_v(t)=0\).
Applying \cref{eq:parity-detector} gives zero incident quadratic sum
at every vertex.
\Cref{lem:four-colour-parity} now gives even degree in every
colour class. Finally, \cref{lem:colouring-criterion} decomposes
the colour classes into compatible circuits.
\end{proof}

\section{A consequence for cycle double covers}\label{sec:cdc}

For this consequence, a \emph{cycle} means an even subgraph and need
not be connected. A \(k\)-\emph{cycle double cover} is a family of at
most \(k\) cycles in which every edge occurs exactly twice.
A circuit is \emph{dominating} if every edge of the graph has at
least one endpoint on it.
The cycle double cover construction of OpenAI
\citep{OpenAI2026CDCProof} gives eight even subgraphs indexed by
\(\Ftwo^3\). In the presence of a prescribed dominating circuit,
\cref{thm:sabidussi} improves the bound and retains that circuit
as a member.

\Needspace{12\baselineskip}
\begin{cor}\label{cor:dominating-circuit}
Let \(H\) be a finite cubic graph and let \(C\) be a dominating
circuit of \(H\). Then \(H\) has a \(5\)-cycle double cover
with \(C\) as a member.
\end{cor}

\begin{proof}
Every vertex of \(C\) has exactly one incident edge outside \(C\).
Since \(C\) is dominating, every vertex outside \(C\) has all three
neighbours on \(C\). Hence each component of \(H-E(C)\) is either
a single edge with both ends on \(C\), or a star with its three
leaves on \(C\). Contract these components, retaining the edges
of \(C\), to obtain a multigraph \(G\)
\citep[pp.~236--237]{Fleischner1984}. The contracted vertices have
degree four or six, and traversing \(C\) once gives an Euler tour
\(T\) of \(G\). Apply \cref{thm:sabidussi} to obtain a four-colouring
\(\chi\) of \(E(G)=E(C)\).

Fix a colour \(z\in\Ftwo^2\). At every vertex of \(C\), the two
edges of \(C\) are consecutive in \(T\) and have different colours.
For a component \(K\) of \(H-E(C)\), the number of its vertices
incident with a \(z\)-coloured edge of \(C\) is therefore the degree
of the contracted vertex in colour \(z\). This number is even and
at most three, so it is zero or two.

Let \(D_z\) consist of the \(z\)-coloured edges of \(C\), together
with the unique path in each component \(K\) joining its two such
vertices, whenever they exist. At a vertex of \(C\), the degree in
\(D_z\) is two if that vertex has an incident \(z\)-coloured edge
of \(C\), and zero otherwise. A vertex outside \(C\) has degree two
if one of the added paths passes through it, and zero otherwise.
Thus \(D_z\) is an even subgraph of \(H\).

Every edge of \(C\) belongs to exactly one \(D_z\). For an edge
\(f\) outside \(C\), choose an end \(u\) on \(C\). The edge \(f\)
belongs to \(D_z\) exactly when \(u\) is incident with a
\(z\)-coloured edge of \(C\), since \(u\) is a leaf of its
component. The two edges of \(C\) at \(u\) have different colours,
so \(f\) belongs to exactly two of the \(D_z\). Adding \(C\) gives
the second covering of its own edges, hence the required
\(5\)-cycle double cover.
\end{proof}

\enlargethispage{2\baselineskip}

\end{document}